\documentclass[letterpaper, 10 pt, conference]{ieeeconf}  
\IEEEoverridecommandlockouts                              

\usepackage{cite}
\usepackage{amsfonts}
\usepackage{algorithmic}
\usepackage{graphicx}
\usepackage{textcomp}
\usepackage{graphics} 
\usepackage{amsmath} 
\usepackage{amssymb}  
\usepackage{euscript}
\usepackage{enumerate}
\usepackage{subcaption}

\newtheorem{defi}{Definition}[section]

\newtheorem{lem}[defi]{Lemma}
\newtheorem{prop}[defi]{Proposition}
\newtheorem{thm}[defi]{Theorem}
\newtheorem{rem}[defi]{Remark}
\newtheorem{cor}[defi]{Corollary}

\title{\LARGE \bf A Risk-Aware Control: \\Integrating Worst-Case CVaR with Control Barrier Function}
\author{Masako Kishida, \IEEEmembership{Senior Member, IEEE}
\thanks{This work was supported by JST, PRESTO Grant Number JPMJPR22C3, Japan.}
\thanks{M. Kishida is with the National Institute of Informatics, Tokyo, Japan (e-mail: kishida@nii.ac.jp). }
}

\begin{document}

\maketitle
\thispagestyle{empty}
\pagestyle{empty}

\begin{abstract}
This paper proposes a risk-aware control approach to enforce safety for discrete-time nonlinear systems subject to stochastic uncertainties.
We derive some useful results on the worst-case Conditional Value-at-Risk (CVaR) and define a discrete-time risk-aware control barrier function using the worst-case CVaR.
On this basis, we present optimization-based control approaches that integrate the worst-case CVaR into the control barrier function, taking into account both safe set and tail risk considerations. 
In particular, three types of safe sets are discussed in detail: half-space, polytope, and ellipsoid. It is shown that control inputs for the half-space and polytopic safe sets can be obtained via quadratic programs, while control inputs for the ellipsoidal safe set can be computed via a semidefinite program. 
Through numerical examples of an inverted pendulum, we compare its performance with existing methods and demonstrate the effectiveness of our proposed controller.
\end{abstract}

\section{Introduction}\label{sec:intro}

Safety-critical systems such as autonomous vehicles, aerospace vehicles, and medical devices require impeccable reliability due to the potentially catastrophic consequences of their failure. As these systems become more prevalent, ensuring their consistent safety in the face of uncertainty becomes more important and challenging. 

Control Barrier Functions (CBFs) \cite{WieA07} have been increasingly used to ensure the safety of various systems such as robotics \cite{RauKH16,AgrS17}, \cite{WanAE17}, spacecraft \cite{BreP22}, and automotive systems \cite{XuGT18, SeoLB22}.
Essentially, the CBF provides guarantees that the system's state will remain within a given safe set for all future times, given that the control input satisfies certain conditions, allowing us to design control systems that satisfy safety requirements.
While early CBF-based control approaches rarely addressed uncertainties,
 recent studies have begun to consider them \cite{NguS16,XuTG15,KolA19,TayA20,SanDC21,Cla21,WanMS21,ChoCT20,CosCT23}.
For example, bounded model uncertainties are considered in the control-Lyapunov-function control-barrier-function quadratic program in \cite{NguS16},
parametric model uncertainties are considered by introducing adaptive CBF in \cite{TayA20}, and
stochastic uncertainties are considered and treated using expectations in CBF \cite{CosCT23}, just to name a few.
However, these approaches may still have limited applicability in safety-critical systems. They overlook risks that could lead to severe consequences when safety-critical systems operate under uncertainty.

To minimize the risk that results in a severe loss, the development of advanced risk-aware control approaches becomes essential. 
The Conditional Value-at-Risk (CVaR) \cite{RocU00,ZymKR13-b} is a risk measure that is defined as the conditional expectation of losses exceeding a certain threshold, thus quantifying the tail risk. CVaR was first used in the finance \cite{RocU00, ZhuF09}, and now it is also used in the controls  \cite{SamY18,MulMS18,ChaL22,HakY22,SurM22}.
However, the computation of CVaR requires the exact knowledge of the uncertainty probability distribution, which may limit its practical use.
The worst-case CVaR  \cite{ZhuF09, RocU00, HuaZFF08}, on the other hand, does not require the exact knowledge of the uncertainty probability distributions, but considers the maximum risk over a set of possible uncertainty distributions, which enhances its applicability in practice. Moreover, it is known that the computation of the worst-case CVaR can be expressed as a quadratic problem for common cases  \cite{ZymKR13-b}, \cite{ZymKR13}. 
 The use of the worst-case CVaR in control problems can be seen, for example, in \cite{GavHK12,JinX21, KisC23, Kis23}.

The primary objective of this paper is to design a risk-aware control approach to enforce safety for discrete-time nonlinear systems subject to stochastic uncertainties. This is achieved by integrating the worst-case CVaR into the control barrier function.
The main contributions of the paper are threefold:
1) The derivation of useful results related to the worst-case CVaR,
2) The introduction of a discrete-time risk-aware control barrier function definition using the worst-case CVaR, and
3) The formulation of optimization-based control approaches that integrate the worst-case CVaR into CBF for three types of safe sets; half-space and polytopic (as quadratic programs) and ellipsoidal (as a semidefinite program).
A paper that is closely related to this paper is \cite{SinAA23}, which presents a higher-level discussion of risk-aware CBF along with specific examples of expectation and CVaR as risk measures. 
Our paper focuses on the use of the worst-case CVaR and provides concrete computational formulation for implementation.
Thus, the main contributions mentioned above are unique to our paper.

The remainder of the paper is organized as follows. 
Section \ref{sec:pre} presents  the notation and definitions and results on the worst-case CVaR. 
After showing the system model we consider in Section \ref{sec:model}, 
 Section \ref{sec:CBF} presents the definition of discrete-time risk-aware control barrier function and obtains the safety constraints to be satisfied by the control input for three types of safe sets, 
 which is followed by the controller design in Section \ref{sec:cont}.
The performance of these proposed controllers is illustrated and compared with an existing standard controller in Section \ref{sec:ex}, and 
 Section \ref{sec:conc} concludes the paper.


\section{Preliminaries}\label{sec:pre}

\subsection{Notation}
The sets of real numbers, real vectors of length $n$, and real matrices
of size $n \times m$ are denoted by $\mathbb{R}$, $\mathbb{R}^n$,  and $\mathbb{R}^{n\times m}$, respectively. 
For $M\in \mathbb{R}^{n\times n}$, $M \succ 0$ indicates $M$ is positive definite. 
$M^\top$ denotes the transpose of a real matrix $M$ and $\text{Tr}(M)$ denotes the trace of $M$.
$I_n$ denotes the identity matrix of size $n$. 
For a vector $v\in \mathbb{R}^n$, $\|v\|$ denotes the Euclidean norm. 
\subsection{Conditional Value-at-Risk}
Let $\mu \in\mathbb{R}^n$ be the mean and $\Sigma \in\mathbb{R}^{n\times n}$ be the covariance matrix of the random vector  $\xi \in\mathbb{R}^n$ under the true distribution $\mathbb{P}$, which is the probability law of $\xi$. 
Thus, it is implicitly assumed that the random vector $\xi$ has finite second-order moments.
Let $\mathcal{P}$ denote the set of all probability distributions on $\mathbb{ R}^n$ that have the same first- and second-order moments as $\mathbb{P}$, i.e.,
\begin{align*} 
 \mathcal{P}= \left\{\mathbb{P} : \mathbb{E}_{\mathbb{P}}\left[ \begin{bmatrix}\xi_i\\ 1\end{bmatrix}\begin{bmatrix}\xi_j\\ 1\end{bmatrix}^{\!\top} \right]
=  \begin{bmatrix}  \Sigma  \delta_{ij}  & 0 \\ 0^\top & 1 \end{bmatrix}, \forall i, j\right\}.
\end{align*}
Here $\delta_{ij}$ denotes the Kronecker delta and $\mathbb{E}_{\mathbb{P}}[ \cdot]$ denotes the expectation with respect to $\mathbb{P}$.
The true underlying probability measure $\mathbb{P}$ is not known exactly, but it is known that  $\mathbb{P}\in \mathcal{P}$. 

\begin{defi}[Conditional Value-at-Risk \cite{RocU00,ZymKR13-b}]
For a given measurable loss function $L : \mathbb{R}^n \rightarrow  \mathbb{R}$, a probability distribution $\mathbb{P}$ on $\mathbb{R}^n$ and a level 
 $\varepsilon \in (0, 1)$, the CVaR at $\varepsilon$ with respect to $\mathbb{P}$ is defined as
 \begin{align*}
\mathbb{P}\text{-CVaR}_{\varepsilon}[L({\xi})] =\inf_{\beta \in \mathbb{R}} \left\{ \beta + \frac{1}{\varepsilon}\mathbb{E}_{\mathbb{P}}[(L({\xi})-\beta)^+]\right\}.
\end{align*}
\end{defi}
CVaR is the conditional expectation of loss above the ($1-\varepsilon$)-quantile of the loss function \cite{ZymKR13-b} and quantifies the tail risk.

The worst-case CVaR  is the supremum of CVaR over a given set of probability distributions as defined below: 
\begin{defi}[Worst-case CVaR \cite{ZymKR13-b}]
The worst-case CVaR over $\mathcal{P}$ is given by
 \begin{align*}
\sup_{\mathbb{P}\in \mathcal{P}}\mathbb{P}\text{-CVaR}_{\varepsilon}[L({\xi})]
=\inf_{\beta \in \mathbb{R}} \left\{ \beta + \frac{1}{\varepsilon}\sup_{\mathbb{P}\in \mathcal{P}}\mathbb{E}_{\mathbb{P}}[(L({\xi})-\beta)^+]\right\}.
\end{align*}
\end{defi}
Here, the exchange between the supremum and infimum is justified by the stochastic saddle point theorem \cite{ShaK02}.

One reason that the (worst-case) CVaR is widely used for risk assessment is its mathematically attractive properties of coherency.
\begin{prop}[Coherence properties \cite{ZhuF09,Art99}]\label{prop:coh}
The worst-case CVaR is a coherent risk measure, i.e., it satisfies the following properties:
Let $L_1 = L_1(\xi)$ and $L_2 = L_2(\xi)$ be two measurable loss functions.
\begin{itemize}
\item Sub-additivity: For all $L_1$ and $L_2$,
\begin{align*}
&\sup_{\mathbb{P}\in \mathcal{P}}\mathbb{P}\text{-CVaR}_{\varepsilon}[L_1+L_2]\\
 & \leq \sup_{\mathbb{P}\in \mathcal{P}}\mathbb{P}\text{-CVaR}_{\varepsilon}[L_1] +
\sup_{\mathbb{P}\in \mathcal{P}}\mathbb{P}\text{-CVaR}_{\varepsilon}[L_2]; 
\end{align*}
\item Positive homogeneity: For a positive constant $c_1>0$, 
\begin{align*}
\sup_{\mathbb{P}\in \mathcal{P}}\mathbb{P}\text{-CVaR}_{\varepsilon}[c_1L_1] =c_1 \sup_{\mathbb{P}\in \mathcal{P}}\mathbb{P}\text{-CVaR}_{\varepsilon}[L_1]; 
\end{align*}
\item Monotonicity: If $L_1\leq L_2$ almost surely, 
\begin{align*}
\sup_{\mathbb{P}\in \mathcal{P}}\mathbb{P}\text{-CVaR}_{\varepsilon}[L_1] \leq
\sup_{\mathbb{P}\in \mathcal{P}}\mathbb{P}\text{-CVaR}_{\varepsilon}[L_2];
\end{align*}
\item Translation invariance: For a constant $c_2$,
\begin{align*}
\sup_{\mathbb{P}\in \mathcal{P}}\mathbb{P}\text{-CVaR}_{\varepsilon}[L_1+c_2] =\sup_{\mathbb{P}\in \mathcal{P}}\mathbb{P}\text{-CVaR}_{\varepsilon}[L_1] +
c_2.
\end{align*}
\end{itemize}
\end{prop}

Another reason that the worst-case CVaR is used for risk assessment due to the fact that it can be computed efficiently for some special cases that appear often.
If $L(\xi)$ is quadratic with respect to $\xi$, then the worst-case CVaR can be computed by a semidefinite program.
Let the second-order moment matrix of $\xi $ by 
\begin{align}\label{eq:Ome}
\Omega &= \begin{bmatrix}\Sigma + \mu \mu^\top & \mu \\ \mu^\top & 1 \end{bmatrix}. 
\end{align}
\begin{lem}[Quadratic function  \cite{ZymKR13-b}, \cite{ZymKR13}] \label{lem:CVaR_quadratic}
Let 
 \begin{align}
 L({\xi}) = \xi^\top P \xi + 2q^\top \xi + r,
 \end{align}
where $P \in \mathbb{S}^n$, $q\in \mathbb{R}^n$ and $r\in \mathbb{R}$.
Then 
 \begin{align} \begin{aligned}\label{eq:sdp}
\sup_{\mathbb{P}\in \mathcal{P}}\mathbb{P}\text{-CVaR}_{\varepsilon}[L({\xi})] =
&\inf_{ \beta} \left\{\beta + \frac{1}{\varepsilon}\text{Tr}(\Omega N): \right.\\
& N \succcurlyeq 0, \\
&\left. N-\left[\begin{array}{cc}P &q \\ q^\top & r-\beta \end{array}\right] \succcurlyeq 0\right\}.
\end{aligned}\end{align}
\end{lem}

If the mean of the random vector  $\xi$ is zero, 
Lemma \ref{lem:CVaR_quadratic}  leads to the following useful results.
\begin{lem}[A property of $L(\xi)$]  \label{lem:CVaR_cross}
Suppose $\mu=0$. Let
 \begin{align*}
 L_1({\xi}) = \xi^\top P \xi + 2q^\top \xi + r,\\
 L_2({\xi}) = \xi^\top P \xi - 2q^\top \xi + r,
  \end{align*}
where $P \in \mathbb{S}^n$, $q\in \mathbb{R}^n$ and $r\in \mathbb{R}$. Then it holds that 
\begin{align}
\sup_{\mathbb{P}\in \mathcal{P}}\mathbb{P}\text{-CVaR}_{\varepsilon}[  L_1({\xi})] =\sup_{\mathbb{P}\in \mathcal{P}}\mathbb{P}\text{-CVaR}_{\varepsilon}[ L_2({\xi})]. \label{eq:CVaR_cross}
\end{align}
\end{lem}
\begin{proof}
Let  $N$ and $\bar{N}$ solve \eqref{eq:sdp} for $\sup_{\mathbb{P}\in \mathcal{P}}\mathbb{P}\text{-CVaR}_{\varepsilon}[ L_1({\xi})] $ and $\sup_{\mathbb{P}\in \mathcal{P}}\mathbb{P}\text{-CVaR}_{\varepsilon}[L_2({\xi})]$. 
Then, they satisfy
 \begin{align*}
N =\left[\begin{matrix} I_n & 0 \\0 & -1\end{matrix}\right] \bar{N}\left[\begin{matrix}I_n & 0 \\0 & -1\end{matrix}\right]. 
\end{align*}
Since $\Omega$ in \eqref{eq:Ome} is block diagonal for $\mu=0$,  the signs of off-diagonal blocks do not enter the objective function, the claim follows.
\end{proof}

\begin{cor} \label{cor:CVaR_cross2}
Suppose $\mu=0$ and $\xi\in \mathbb{R}$. Then it holds that 
\begin{align}
\sup_{\mathbb{P}\in \mathcal{P}}\mathbb{P}\text{-CVaR}_{\varepsilon}[  \xi ] =\sup_{\mathbb{P}\in \mathcal{P}}\mathbb{P}\text{-CVaR}_{\varepsilon}[-\xi]\geq 0. \label{eq:CVaR_cross2}
\end{align}
\end{cor}

\begin{lem}[A bound on linear $L(\xi)$] \label{lem:CVaR_bd}
Suppose $\mu=0$. Then,
\begin{align}
\sup_{\mathbb{P}\in \mathcal{P}}\mathbb{P}\text{-CVaR}_{\varepsilon}[ q^\top \xi] 
\leq \sum_{i=1}^n |q_i| \sup_{\mathbb{P}\in \mathcal{P}}\mathbb{P}\text{-CVaR}_{\varepsilon}[  \xi_i] 
\end{align}
where $q\in \mathbb{R}^n$. 
\end{lem}
\begin{proof}
From Proposition \ref{prop:coh}, it follows that
\begin{align}\begin{aligned}
\sup_{\mathbb{P}\in \mathcal{P}}\mathbb{P}\text{-CVaR}_{\varepsilon}[ q^\top \xi ]
&=\sup_{\mathbb{P}\in \mathcal{P}}\mathbb{P}\text{-CVaR}_{\varepsilon}\left[\sum_{i=1}^n q_i \xi_i\right] \\
&\leq \sum_{i=1}^n \sup_{\mathbb{P}\in \mathcal{P}}\mathbb{P}\text{-CVaR}_{\varepsilon}[ q_i \xi_i] \\
&= \sum_{i=1}^n |q_i| \sup_{\mathbb{P}\in \mathcal{P}}\mathbb{P}\text{-CVaR}_{\varepsilon}[ \text{sign}(q_i) \xi_i] \\
&= \sum_{i=1}^n |q_i| \sup_{\mathbb{P}\in \mathcal{P}}\mathbb{P}\text{-CVaR}_{\varepsilon}[  \xi_i].
\end{aligned}\end{align}
The last equality follows from Lemma \ref{lem:CVaR_cross}. 
\end{proof}

\begin{rem}
For a vector $v = \left[ v_1, \cdots, v_n\right]^\top\in \mathbb{R}^n$,  we define element-wise inequality, element-wise absolute value and element-wise worst-case CVaR by
\begin{align}
v\geq 0 \ &\Leftrightarrow \  v_1\geq 0, \cdots, v_n \geq 0, \\
|v|
&=\left[\begin{matrix}
|v_1| \\
\vdots\\
|v_n|
\end{matrix}\right],\\
\sup_{\mathbb{P}\in \mathcal{P}}\mathbb{P}\text{-CVaR}_{\varepsilon}[ v] 
&=\left[\begin{matrix}
\sup_{\mathbb{P}\in \mathcal{P}}\mathbb{P}\text{-CVaR}_{\varepsilon}[ v_1] \\
\vdots\\
\sup_{\mathbb{P}\in \mathcal{P}}\mathbb{P}\text{-CVaR}_{\varepsilon}[ v_n] 
\end{matrix}\right],
\end{align}
respectively.
\end{rem}


\section{Discrete-time Control-affine System}\label{sec:model}
This paper deals with a discrete-time control-affine system in the form of
\begin{align}
x_{t+1} = f(x_t) +g(x_t)u_t+w_t, \ x(0) = x_0\in \mathcal{C}, \label{eq:sys}
\end{align}
where $x_t\in \mathbb{R}^n$, $u_t \in \mathbb{R}^m$ and $w_t \in  \mathbb{R}^n$ denote the state, the control input, and disturbance of the system at discrete time instant $t$, respectively, $f:\mathbb{R}^n \rightarrow \mathbb{R}^n$ and $g:\mathbb{R}^n\times \mathbb{R}^m \rightarrow \mathbb{R}$ are continuous functions, and $ \mathcal{C}$ is a safe set.

It is assumed that $w_t$ is the sequence of independent and identically distributed random variables with 
a distribution having zero mean $\mu_w=0$ and finite covariance $\Sigma_w \succ 0$ for all $t$.  Although the precise probability measure $\mathbb{P}$ is not known exactly, it is known that  $\mathbb{P}\in \mathcal{P}$. This set is defined as:
\begin{align*}
 \mathcal{P}= \left\{\mathbb{P}: \mathbb{E}_{\mathbb{P}}\left[ \left[\begin{array}{c}w_i\\ 1\end{array} \right]\left[\begin{array}{c}w_j\\ 1\end{array} \right]^{\top}\right]
=  \left[\begin{array}{cc} \Sigma_w \delta_{ij}  & 0\\ 0^\top & 1 \end{array}\right], \forall i, j\right\}.
\end{align*}
Thus,  the second-order moment matrix of $w_t$ is given by
\begin{align}
\Omega_w = \left[\begin{array}{cc} \Sigma_w  & 0\\ 0^\top & 1 \end{array}\right].
\label{eq:omega_w}
\end{align}


\section{Control Barrier Function}\label{sec:CBF}
We extend the discrete-time control barrier function described in \cite{CosCT23} to accommodate risk considerations.

We first define the safe set $\mathcal{C}$ as the superlevel set
of a continuously differentiable function $h : \mathbb{R}^n \rightarrow \mathbb{R}^p$
\begin{align}\label{eq:safeset}
\mathcal{C} = \{x \in  \mathbb{R}^n: h(x)\geq 0\}.
\end{align}
We assume $p=1$ except for Section \ref{sec:poly}.

With $\mathcal{C}$, the discrete-time risk-aware control barrier function is defined below.
\begin{defi}[Discrete-time Risk-Aware Control Barrier Function]\label{defi:CBF}
The function $h$ is a discrete-time risk-aware control barrier function for \eqref{eq:sys} on $\mathcal{C}$ if there exists an $\alpha \in [0,1]$ such that for each $x \in  \mathbb{R}^n$,
there exists a $u \in  \mathbb{R}^m$ such that:
\begin{align}\label{eq:CBF}
\sup_{\mathbb{P}\in \mathcal{P}}\mathbb{P}\text{-CVaR}_{\varepsilon}[-h(x^+) ]\leq - \alpha h (x).
\end{align}
\end{defi}

\begin{rem}
The condition \eqref{eq:CBF} is sufficient for  
\begin{align}
\inf_{\mathbb{P}\in \mathcal{P}}\mathbb{P}[\alpha h (x)-h(x^+)\leq 0]\geq 1- \varepsilon. \label{eq:CBF2}
\end{align}
Moreover, if $-h(x^+)$ is concave in $w$ or quadratic in $w$, then the condition \eqref{eq:CBF} is a necessary and sufficient condition for 
\eqref{eq:CBF2}   \cite{ZymKR13-b}.
Thus, roughly speaking, the condition \eqref{eq:CBF} aims at having $h(x^+)\geq \alpha h(x)$ with high probability.
\end{rem}

\begin{rem}
It's worth noting that the condition \eqref{eq:CBF} is different from
\begin{align}
\sup_{\mathbb{P}\in \mathcal{P}}\mathbb{P}\text{-CVaR}_{\varepsilon}[h(x^+) ]\geq  \alpha h (x)
\end{align}
due to the coherent property in Proposition \ref{prop:coh}. 
Thus different from the definition in  \cite{SinAA23}.
\end{rem}

In the following, we will see how the safety constraint \eqref{eq:CBF} can be simplified so that the integrated control problem can be solved efficiently. 
\subsection{Half-space safe set}\label{sec:hs}

Consider a scenario where the safe set is a half-space. In this case, it can be represented by an affine function $h(x)$:
\begin{align}\begin{aligned} \label{eq:hs-set}
\mathcal{C}_{\text{hs}} &= \{x \in  \mathbb{R}^n: h(x)= q^\top x+r  \geq 0\}.
\end{aligned}\end{align}

Given this representation, the constraint \eqref{eq:CBF} can be expressed as a linear function of $u$ as follows:
\begin{thm}\label{thm:hs}
If $h(x) = q^\top x+r$, $q\in\mathbb{R}^n$ and $r\in\mathbb{R}$, then the constraint \eqref{eq:CBF} holds if and only if
\begin{align}
-q^\top g(x)u &\leq \phi(x), \label{eq:affine}
\end{align}
where
\begin{align}
\phi (x) =q^\top ( f(x)-\alpha  x) +(1-\alpha)r - \sup_{\mathbb{P}\in \mathcal{P}}\mathbb{P}\text{-CVaR}_{\varepsilon}[q^\top w ].
\end{align}
\end{thm}
\begin{proof}
By substitution, the right-hand-side of \eqref{eq:CBF} becomes
\begin{align}
-\alpha h(x) = -\alpha ( q^\top x+r).
\end{align}
On the other hand, using Proposition \ref{prop:coh} and Lemma \ref{lem:CVaR_cross}, the left-hand-side of \eqref{eq:CBF} becomes
\begin{align}\begin{aligned}
&\sup_{\mathbb{P}\in \mathcal{P}}\mathbb{P}\text{-CVaR}_{\varepsilon}[-h(x^+) ] \\
=& \sup_{\mathbb{P}\in \mathcal{P}}\mathbb{P}\text{-CVaR}_{\varepsilon}[-( q^\top x^++r) ]\\
=&\sup_{\mathbb{P}\in \mathcal{P}}\mathbb{P}\text{-CVaR}_{\varepsilon}[-(q^\top( f(x)+g(x)u+w)+r) ]\\
=&-(q^\top( f(x)+g(x)u)+r)+ \sup_{\mathbb{P}\in \mathcal{P}}\mathbb{P}\text{-CVaR}_{\varepsilon}[-q^\top w ]\\
=&-(q^\top( f(x)+g(x)u)+r)+ \sup_{\mathbb{P}\in \mathcal{P}}\mathbb{P}\text{-CVaR}_{\varepsilon}[q^\top w ].
\end{aligned}\end{align}
Thus, the constraint \eqref{eq:CBF} simplifies as 
\begin{align}\begin{aligned}
-q^\top g(x)u &\leq q^\top f(x)+r- \sup_{\mathbb{P}\in \mathcal{P}}\mathbb{P}\text{-CVaR}_{\varepsilon}[q^\top w ]-\alpha (q^\top x+r)\\
& = \phi(x).
\end{aligned}\end{align}
\end{proof}
\begin{rem}
As $\varepsilon$ approaches to 1, $\sup_{\mathbb{P}\in \mathcal{P}}\mathbb{P}\text{-CVaR}_{\varepsilon}[q^\top w ]$ approaches to 0, thus the condition \eqref{eq:affine} approaches to
\begin{align}
-q^\top g(x)u &\leq q^\top f(x)+r- \alpha (q^\top x+r),
\end{align}
which disregards the effect of uncertainties, or focusing on the performance of the expected value.
\end{rem}
\subsection{Polytopic safe set}\label{sec:poly}
In scenarios where the safe set takes the form of a polytope, we may also explicitly express the condition for $u$ to satisfy \eqref{eq:CBF}. 
In this case, the safe set is defined as the intersection of the superlevel set of 
affine functions $h_i(x) = q_i^\top x + r_i$ for $i = 1,\cdots, m$ for $q_i\in\mathbb{R}^n$ and $r_i\in\mathbb{R}$:
\begin{align}\begin{aligned} \label{eq:poly-set}
\mathcal{C}_{\text{poly}} &= \{x \in  \mathbb{R}^n: h_i (x) \geq 0, i = 1,\cdots, m \}\\
&=\{x \in  \mathbb{R}^n: h(x) = Q^\top x + r \geq 0 \},
\end{aligned}\end{align}
where
\begin{align}\begin{aligned}
Q = \left[\begin{matrix}q_1 & \cdots& q_m \end{matrix}\right]\in\mathbb{R}^{n\times m}, \quad 
r =  \left[\begin{matrix}r_1& \cdots&  r_m\end{matrix}\right]^\top \in\mathbb{R}^{m}.
\end{aligned}\end{align}
Similarly to Theorem \ref{thm:hs}, we have the following result.
\begin{thm}\label{thm:poly}
If $h(x) = Q^\top x+r$ with $Q\in\mathbb{R}^{n\times m}$ and $r\in\mathbb{R}^m$, then the constraint \eqref{eq:CBF} holds if and only if
\begin{align}\begin{aligned} \label{eq:poly}
-Q^\top g(x)u &\leq \phi(x),
\end{aligned}\end{align}
where
\begin{align}
\phi (x) =Q^\top ( f(x)-\alpha  x) +(1-\alpha) r- \sup_{\mathbb{P}\in \mathcal{P}}\mathbb{P}\text{-CVaR}_{\varepsilon}[Q^\top w ].
\end{align}
\end{thm}
\begin{proof}
The proof directly follows from the proof of Theorem \ref{thm:hs}. 
\end{proof}

\begin{rem}
Similar to earlier observations, as $\varepsilon$ approaches to 1, 
clearly the condition \eqref{eq:poly} approaches to
\begin{align}
-Q^\top g(x)u &\leq Q^\top ( f(x)-\alpha  x) +(1-\alpha) r.
\end{align}
\end{rem}

\begin{rem}
This case, as well as the following two case, there is no guarantee that there is a feasible input that satisfies the constraint \eqref{eq:poly}.
If infeasible, we may want to introduce some penalty for the violation of this constraint. This is discussed in the controller design in Section \ref{sec:cont}.
\end{rem}
\subsection{Ellipsoidal safe set}\label{sec:ell}
Another tractable scenario arises when the safe set takes the form of an ellipsoid. 
In such situations, the safe set is expressed by using a positive definite matrix $E=E^\top \succ 0$:
\begin{align}\begin{aligned} \label{eq:ell-set}
\mathcal{C}_{\text{ell}} &= \{x \in  \mathbb{R}^n: x^\top E x \leq r \}\\
&=\{x \in  \mathbb{R}^n: h(x) = - x^\top E x + r \geq 0 \}.
\end{aligned}\end{align}
In this case, a sufficient condition for the constraint \eqref{eq:CBF} to be satisfied can be expressed using a quadratic function of $u$ as follows:
\begin{thm}\label{thm:ell}
Let $h(x) = -x^\top E x+ r$, $E\in\mathbb{R}^{n\times n}$ with $E= E^\top\succ 0$ and $r\in\mathbb{R}$. Also define $\bar{u} = [u^\top, v^\top]^\top$ for the control input $u\in \mathbb{R}^m$ and some variable $v\in \mathbb{R}^n$. Then, the constraint \eqref{eq:CBF} holds if 
\begin{align}\begin{aligned}  \label{eq:ell-const}
1) \quad & \bar{u}^\top \bar{H}(x)\bar{u} + \bar{q}^\top(x) \bar{u}+ \bar{r}(x)\leq 0 \text{ and } \\
2) \quad &\bar{A}(x)\bar{u} \leq 0,
\end{aligned}\end{align}
where
\begin{align}\begin{aligned} \label{eq:ell-const2}
\bar{H}(x) & = \left[ \begin{matrix}g(x)^\top E g(x)& 0\\ 0 & 0 \end{matrix}\right] \\
\bar{q}(x) & =2 \left[ \begin{matrix} g(x)^\top E f(x) \\\sup_{\mathbb{P}\in \mathcal{P}}\mathbb{P}\text{-CVaR}_{\varepsilon}[ w] \end{matrix}\right]\\
\bar{r}(x) & =\sup_{\mathbb{P}\in \mathcal{P}}\mathbb{P}\text{-CVaR}_{\varepsilon}[w^\top E w+(2Ef(x) )^\top w  ]  \\
&  +f(x)^\top E f(x)-r -\alpha (x^\top Ex - r)\\
\bar{A}(x) & = \left[ \begin{matrix}  Eg(x) & -I\\ - Eg(x) & -I \end{matrix}\right]. 
\end{aligned}\end{align}
\end{thm}

\begin{proof}
Using Proposition \ref{prop:coh} and Lemmas \ref{lem:CVaR_cross} and \ref{lem:CVaR_bd}, it follows that
\begin{align}\begin{aligned}
&\sup_{\mathbb{P}\in \mathcal{P}}\mathbb{P}\text{-CVaR}_{\varepsilon}[-h(x^+) ]+ \alpha h (x)\\
=&\sup_{\mathbb{P}\in \mathcal{P}}\mathbb{P}\text{-CVaR}_{\varepsilon}[(f(x)+g(x)u+w)^\top E(f(x)+g(x)u+w)\\
& \qquad-  r] - \alpha (x^\top Ex - r)\\
= &\sup_{\mathbb{P}\in \mathcal{P}}\mathbb{P}\text{-CVaR}_{\varepsilon}[w^\top E w+2(f(x)+g(x)u )^\top Ew  ]\\
 &\qquad+(f(x)+g(x)u )^\top E (f(x)+g(x)u )^\top \\
  &\qquad -  r - \alpha (x^\top Ex  - r)\\
\leq &\sup_{\mathbb{P}\in \mathcal{P}}\mathbb{P}\text{-CVaR}_{\varepsilon}[w^\top E w+(2Ef(x) )^\top w  ]\\
 &\qquad+ \sup_{\mathbb{P}\in \mathcal{P}}\mathbb{P}\text{-CVaR}_{\varepsilon}[(2Eg(x)u )^\top w  ]\\
 &\qquad+ u^\top g^\top(x) E g(x)u + (2g^\top(x) E f(x) )^\top u  \\
  &\qquad +  f^\top (x) E f(x) - r - \alpha (x^\top Ex - r)\\
\leq& \sup_{\mathbb{P}\in \mathcal{P}}\mathbb{P}\text{-CVaR}_{\varepsilon}[w^\top E w+(2Ef(x) )^\top w  ]\\
 &\qquad+2v^\top \sup_{\mathbb{P}\in \mathcal{P}}\mathbb{P}\text{-CVaR}_{\varepsilon}[w  ]\\
 &\qquad+ u^\top g^\top(x) E g(x)u + (2g^\top(x) E f(x)  )^\top u  \\
  &\qquad +  f^\top (x) E f(x) - r - \alpha (x^\top Ex  - r),
\end{aligned}\end{align}
where
\begin{align}\begin{aligned}
-v\leq   Eg(x)u   \leq v. \label{eq:remove_abs2}
\end{aligned}\end{align}
Thus, a sufficient condition for 
 \begin{align}\begin{aligned}
&\sup_{\mathbb{P}\in \mathcal{P}}\mathbb{P}\text{-CVaR}_{\varepsilon}[-h(x^+) ]+ \alpha h (x)\leq 0
\end{aligned}\end{align}
is to satisfy
 \begin{align}\begin{aligned}
&\sup_{\mathbb{P}\in \mathcal{P}}\mathbb{P}\text{-CVaR}_{\varepsilon}[w^\top E w+(2Ef(x) )^\top w  ]\\
 &\qquad+2v^\top \sup_{\mathbb{P}\in \mathcal{P}}\mathbb{P}\text{-CVaR}_{\varepsilon}[w  ]\\
 &\qquad+ u^\top g^\top(x) E g(x)u + (2g^\top(x) E f(x)  )^\top u  \\
  &\qquad +  f^\top (x) E f(x) - r - \alpha (x^\top Ex  - r) \leq 0 
\end{aligned}\end{align}
and \eqref{eq:remove_abs2}.
Those can be expressed as in \eqref{eq:ell-const}-\eqref{eq:ell-const2}.
\end{proof}

\begin{rem}
Unlike the previous two cases in Subsections \ref{sec:hs} and \ref{sec:poly}, the obtained condition \eqref{eq:ell-const}-\eqref{eq:ell-const2} is only sufficient for the safety condition \eqref{eq:CBF} is satisfied. This is because we utilize Lemma \ref{lem:CVaR_bd} to pull out the control input outside of the worst-case CVaR. 
Moreover, there is no guarantee that there is a feasible input that satisfies the safety constraint \eqref{eq:ell-const}.
As in the case of Subsection \ref{sec:poly}, if infeasible, we may want to introduce some penalty for the violation of this constraint, which is discussed in Section \ref{sec:cont}.
\end{rem}

\subsection{General safe set}\label{sec:gen}
For the general safe set, which is expressed by using a general function of $h(x)$, there is no simple way to obtain an equivalent or sufficient condition as the safety constraint \eqref{eq:CBF}. 
Similarly to \cite{CosCT23}, this motivates us to consider the relation between  $\sup_{\mathbb{P}\in \mathcal{P}}\mathbb{P}\text{-CVaR}_{\varepsilon}[-h(x^+)]$ and the function
\begin{align}
-h\left(\sup_{\mathbb{P}\in \mathcal{P}}\mathbb{P}\text{-CVaR}_{\varepsilon}[ {x}^+]\right),
\end{align}
which is likely to be more tractable.

Here, we restrict the case that $h$ is cocave in $x$ and that it satisfies
\begin{align}\label{eq:hessian}
-\underline{\sigma} I \leq \nabla^2 h(x), \ \forall x \in  \mathbb{R}^n
\end{align}
for some $\underline{\sigma} \geq 0$.

Let define
\begin{align}\begin{aligned} \label{eq:notation}
\bar{w} &= \sup_{\mathbb{P}\in \mathcal{P}}\mathbb{P}\text{-CVaR}_{\varepsilon}[ w],\\
\bar{x}^+ &= \sup_{\mathbb{P}\in \mathcal{P}}\mathbb{P}\text{-CVaR}_{\varepsilon}[ x^+]= f(x)+g(x)u+\bar{w},\\
z&= x^+ - \bar{x}^+  = w-\bar{w}.
\end{aligned}\end{align}

Then we have the following result.
\begin{lem}\label{lem:approx}It holds that
\begin{align}\begin{aligned} \label{eq:bound}
&\sup_{\mathbb{P}\in \mathcal{P}}\mathbb{P}\text{-CVaR}_{\varepsilon}[-h(x^+)]\leq \\
&-h(\bar{x}^+) +
\sup_{\mathbb{P}\in \mathcal{P}}\mathbb{P}\text{-CVaR}_{\varepsilon}[-\nabla h(\bar{x}^+)^\top z+ \frac{\underline{\sigma}}{2}z^\top z].
\end{aligned}\end{align}
\end{lem}
\begin{proof}
By Taylor's theorem, there exists some $c$ on the line segment connecting $x^+$ and $\bar{x}^+$ such that 
\begin{align}
h(x^+) - h(\bar{x}^+) = \nabla h(\bar{x}^+)^\top z+ \frac{1}{2}z^\top \nabla^2 h(c)z.
\end{align}

Under the assumption \eqref{eq:hessian}, it follows that
\begin{align}\begin{aligned}
-h(x^+)& = - h(\bar{x}^+) -\nabla h(\bar{x}^+)^\top z- \frac{1}{2}z^\top \nabla^2 h(c)z\\
&\leq - h(\bar{x}^+) -\nabla h(\bar{x}^+)^\top z+ \frac{\underline{\sigma}}{2}z^\top z.
\end{aligned}\end{align}
Thus, using Proposition \ref{prop:coh}, \eqref{eq:bound} is obtained.
\end{proof}

\begin{rem}
As $\varepsilon$ approaches to 1, this agrees with the result in \cite{CosCT23}.
\end{rem}

\begin{cor}\label{cor:general}
From Lemma \ref{lem:approx}, a sufficient condition that the constraint \eqref{eq:CBF} is satisfied is
\begin{align}\begin{aligned} \label{eq:general}
 - h(\bar{x}^+)+\sup_{\mathbb{P}\in \mathcal{P}}\mathbb{P}\text{-CVaR}_{\varepsilon}[  -\nabla h(\bar{x}^+)^\top z+ \frac{\underline{\sigma}}{2}z^\top z]\leq - \alpha h (x).
\end{aligned}\end{align}
\end{cor}

\begin{rem}
The control input $u$ enters $h(\bar{x}^+)$ as well as $\nabla h(\bar{x}^+)$. 
Thus, still, the condition for $u$ to satisfy the safety constraint cannot be explicitly expressed. 
Nevertheless, we may check the feasibility of the condition \eqref{eq:general} once $u$ is given using the semidefinite programming:
Here, the substitution of \eqref{eq:notation} leads to 
\begin{align}\begin{aligned}
& \sup_{\mathbb{P}\in \mathcal{P}}\mathbb{P}\text{-CVaR}_{\varepsilon}\left[  -\nabla h(\bar{x}^+)^\top z+ \frac{\underline{\sigma}}{2}z^\top z\right]\\
=&\sup_{\mathbb{P}\in \mathcal{P}}\mathbb{P}\text{-CVaR}_{\varepsilon}\left[  -\nabla h(\bar{x}^+)^\top (w-\bar{w})+ \frac{\underline{\sigma}}{2}(w-\bar{w})^\top (w-\bar{w})\right]\\
=&\sup_{\mathbb{P}\in \mathcal{P}}\mathbb{P}\text{-CVaR}_{\varepsilon}\left[  \frac{\underline{\sigma}}{2} w^\top w -\left( \nabla h(\bar{x}^+) + \underline{\sigma}\bar{w}\right)^\top w 
\right.\\ & \left. \qquad\qquad\qquad\qquad\qquad\qquad\qquad
+ \frac{\underline{\sigma}}{2}\bar{w}^\top\bar{w} + \nabla h(\bar{x}^+)^\top \bar{w}  \right].
\end{aligned}\end{align}
Namely,  $-\nabla h(\bar{x}^+)^\top z+ \frac{\underline{\sigma}}{2}z^\top z$ is quadratic and convex with respect to $w$, thus the worst-case CVaR can be evaluated using semidefinite programming as in Lemma \ref{lem:CVaR_quadratic} once $x$ and $u$ are given.
\end{rem}


\section{Controller Design}  \label{sec:cont}
Given the derived safety constraints, our objective is to design controllers that minimally modify a nominal controller that does not take the safe set into account.
For this purpose, optimization-based control approaches are developed \cite{CosCT23},  \cite{AmeCE19}:
\begin{align} \label{eq:controller}
\begin{aligned} 
u ^*(x) &=  \text{argmin}_{u}  \|u-u_{\text{nom}}(x)\|^2\\
\text{s.t. }\ & \text{risk-aware safety constraint}
\end{aligned}
\end{align}

This controller ensures safety and strives for minimal point-wise divergence from the nominal controller $u_{\text{nom}}$ which does not take the safe set in consideration.

The following are specific optimization problems to obtain control inputs for each of the three safe sets.
We do not discuss the general case of safe set.

\subsubsection{Half-space safe set}\label{sec:hs_c}
If $h(x) = q^\top x+r$, $q\in\mathbb{R}^n$ and $r\in\mathbb{R}$, then using the result of Theorem \ref{thm:hs}, the control input $u ^*(x) $ is obtained by repeatedly solving
\begin{align} \label{eq:hs_c}
\begin{aligned} 
u ^*(x) &=  \text{argmin}_{u}  \|u-u_{\text{nom}}(x)\|^2\\
\text{s.t. }\ & -q^\top g(x)u \leq \phi(x),
\end{aligned}
\end{align}
where
\begin{align}
\phi (x) =q^\top ( f(x)-\alpha  x) +(1-\alpha)r - \sup_{\mathbb{P}\in \mathcal{P}}\mathbb{P}\text{-CVaR}_{\varepsilon}[q^\top w ].
\end{align}

This is a quadratic program and can be solved efficiently.
Moreover, the term $\sup_{\mathbb{P}\in \mathcal{P}}\mathbb{P}\text{-CVaR}_{\varepsilon}[q^\top w ]$ can be precomputed because it does not depend on $x$.
This term serves as a safety margin to avoid the potential risk of violation due to disturbances. This margin becomes smaller as we choose a larger $\varepsilon$, leading to a scenario where the disturbances are effectively disregarded. 

\subsubsection{Polytopic safe set}\label{sec:poly_c}
If $h(x) = Q^\top x+r$, $Q\in\mathbb{R}^{n\times m} $ and $r\in\mathbb{R}^m$, then using the result of Theorem \ref{thm:poly}, the control input is obtained by solving
\begin{align} \label{eq:poly_c}
\begin{aligned} 
u ^*(x) &=  \text{argmin}_{u}  \|u-u_{\text{nom}}(x)\|^2\\
\text{s.t. }\ & -Q^\top g(x)u \leq \phi(x),
\end{aligned}
\end{align}
where
\begin{align}
\phi (x) =Q^\top ( f(x)-\alpha  x) +(1-\alpha)r - \sup_{\mathbb{P}\in \mathcal{P}}\mathbb{P}\text{-CVaR}_{\varepsilon}[Q^\top w ].
\end{align}

As mentioned in Subsection \ref{sec:poly}, this optimization problem \eqref{eq:poly_c} can be infeasible, in which case, one way of revising it is 
\begin{align} \label{eq:poly_c_rev}
\begin{aligned} 
u ^*(x) &=  \text{argmin}_{u}  \|u-u_{\text{nom}}(x)\|^2 +\rho \delta \\
\text{s.t. }\ & -Q^\top g(x)u \leq \phi(x)+\delta, \\
& \delta \geq 0
\end{aligned}
\end{align}
by introducing a parameter $\rho > 0$. Other simple modifications would work as well.

In any case, both optimizations in \eqref{eq:poly_c} and \eqref{eq:poly_c_rev} are quadratic programs and can be solved efficiently.
\subsubsection{Ellipsoidal safe set}\label{sec:ell_c}
If $h(x) = - x^\top E x + r $ with $E= E^\top\succ 0$ and $r\in\mathbb{R}$, then using the result of Theorem \ref{thm:ell}, the control input is obtained by solving
\begin{align} \label{eq:ell_c}
\begin{aligned} 
u ^*(x) &=  \text{argmin}_{u}  \|u-u_{\text{nom}}(x)\|^2\\
\text{s.t. }\ & \bar{u}^\top \bar{H}(x)\bar{u} + \bar{q}^\top(x) \bar{u}+ \bar{r}(x)\leq 0,\\
& \bar{A}(x)\bar{u} \leq 0
\end{aligned}
\end{align}
using the expressions in \eqref{eq:ell-const}.

If infeasible, the optimization in \eqref{eq:ell_c} can be revised in a similar manner as in the polytopic safe sets
\begin{align} \label{eq:ell_c_rev}
\begin{aligned} 
u ^*(x) &=  \text{argmin}_{u}  \|u-u_{\text{nom}}(x)\|^2 +\rho \delta\\
\text{s.t. }\ & \bar{u}^\top \bar{H}(x)\bar{u} + \bar{q}^\top(x) \bar{u}+ \bar{r}(x)\leq \delta,\\
& \bar{A}(x)\bar{u} \leq 0,\\
& \delta \geq 0
\end{aligned}
\end{align}
by introducing a parameter $\rho > 0$.

The constraint is quadratic in this case, thus, the optimization problem \eqref{eq:ell_c_rev} is a quadratically constrained quadratic program.
However, because the objective function is convex and $\bar{H}(x)$ is positive semidefinite, this can be written as a semidefinite program which is tractable.

\section{Numerical Example}\label{sec:ex}

In this section, we evaluate the performance of the controllers developed in Section \ref{sec:cont}, through simulations on an inverted pendulum model. 
The discrete-time dynamics of this model around its upright equilibrium position is given by
\begin{align}\begin{aligned}\label{eq:pendulum}
\left[\begin{matrix}x_{t+1} \\y_{t+1} \end{matrix}\right]& = 
\left[\begin{matrix}x_{t} + y_t \Delta t  \\ y_{t} + \sin(x_t)\Delta t \end{matrix}\right]
+\left[\begin{matrix}0 \\ \Delta t  \end{matrix}\right] u_t + w_t,
 \end{aligned} \end{align}
 where $x_t$ is the angle from the upright position and $y_t$ is the angular velocity, $u_t$ is the control input and $w_t$ is the disturbance.
It is assumed that the disturbance is mean zero and that the covariance is $\Sigma_w= \left[\begin{matrix}0.001^2 & 0 \\ 0 & 0.003^2 \end{matrix}\right]$ and the sampling time $\Delta t = 0.01$. 

To quantify the uncertainty, the worst-case CVaR is used with $\varepsilon = 0.3$.
For the risk-aware safety constraint, $\alpha=0.8$ is chosen.

Take a nominal stabilizing controller 
\begin{align}\begin{aligned}
u_t= -x_t-\sin(x_t)-y_t, \label{eq:nominal_controller} 
 \end{aligned} \end{align}
which is obtained by feedback linearization.

In the followings, all simulations are performed for the duration of time 8, starting from the initial state $[0.3, 0.2]^\top$.

First, we compare the proposed controller with the nominal controller \eqref{eq:nominal_controller} using the three types of safe sets;
\begin{itemize}
\item Half-space: $h(x) = q^\top x + r$ with $q =\left[\begin{matrix}1.125 & 1\end{matrix}\right]^\top$ and $r = 0.075$
\item Polytope: $h(x) = Q^\top x + r$ with $q = \left[\begin{matrix}1.125& 1 \\ 0.5 & 1\end{matrix}\right]$ and $r = \left[\begin{matrix}0.075& 0.1\end{matrix}\right]^\top$
\item Ellipsoid: $h(x) = -x^\top E x + r$ with $E =\left[\begin{matrix}6& -5\\-5& 6\end{matrix}\right] $ and $r =1 $.
\end{itemize}
For the polytopic safe set,  it turns out that the optimization problem \eqref{eq:poly_c} is always feasible, so \eqref{eq:poly_c} is used instead of \eqref{eq:poly_c_rev}. 
For the ellipsoidal safe set, however, the optimization problem \eqref{eq:ell_c} is not always feasible, thus  \eqref{eq:ell_c_rev} is used.
For each of the three types of safe sets, the following three trajectories are compared.
\begin{itemize}
\item Nominal $w=0$: with the nominal controller \eqref{eq:nominal_controller} without disturbance
\item Proposed:  with the proposed risk-aware controller \eqref{eq:hs_c}, \eqref{eq:poly_c} or \eqref{eq:ell_c_rev} with $\rho = 500$ 
subject to Gaussian disturbance with the assumed mean and covariance
\item Proposed $w=0$: with the proposed risk-aware controller \eqref{eq:hs_c}, \eqref{eq:poly_c} or \eqref{eq:ell_c_rev} with $\rho = 500$
without disturbance
\end{itemize}

The results are shown in Figures \ref{fig:hs}-\ref{fig:ell}.

In all cases, the trajectories of the proposed controllers successfully stay inside the safe sets. 

Comparing the cases  ``Nominal $w=0$'' and ``Proposed $w=0$'',  we see that when the states are far from the boundaries of $h= 0$, the trajectories coincide, but as soon as the trajectories approach close to the boundary, there are clear differences; cases of ``Proposed $w=0$'' keep some distances from the boundaries; those distances are in fact, $\sup_{\mathbb{P}\in \mathcal{P}}\mathbb{P}\text{-CVaR}_{\varepsilon}[q^\top w ]$ and $\sup_{\mathbb{P}\in \mathcal{P}}\mathbb{P}\text{-CVaR}_{\varepsilon}[Q^\top w ]$ in the cases of half-space and polytopic safe sets, respectively. The case of an ellipsoidal safe set, the distance is not exactly known because the revised controller  \eqref{eq:ell_c_rev} does not guarantee the satisfaction of the safety constraint.

Comparing the cases  ``Proposed'' and ``Proposed $w=0$'',  we see the effects of the disturbances. 
The trajectories of ``Proposed'' deviate, but stay around ``Proposed $w=0$''. Thanks to the risk-aware design, the trajectories stay inside the safe sets. 

\begin{figure}
\begin{minipage}[b]{0.95\hsize}
 \centering
\includegraphics[width=.98\linewidth, viewport =10 0 430 310, clip]{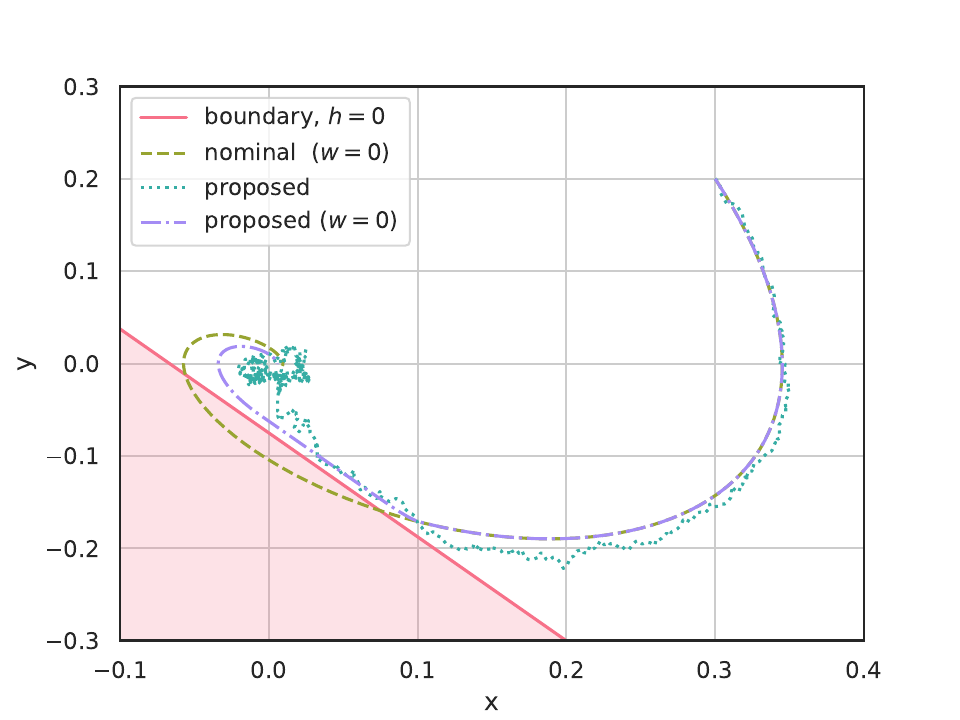}\vspace{-.1in}
\subcaption{Half-space safe set} 
\label{fig:hs}
\end{minipage}
\begin{minipage}[b]{0.95\hsize}
 \centering
  \includegraphics[width=.98\linewidth, viewport =10 0 430 310, clip]{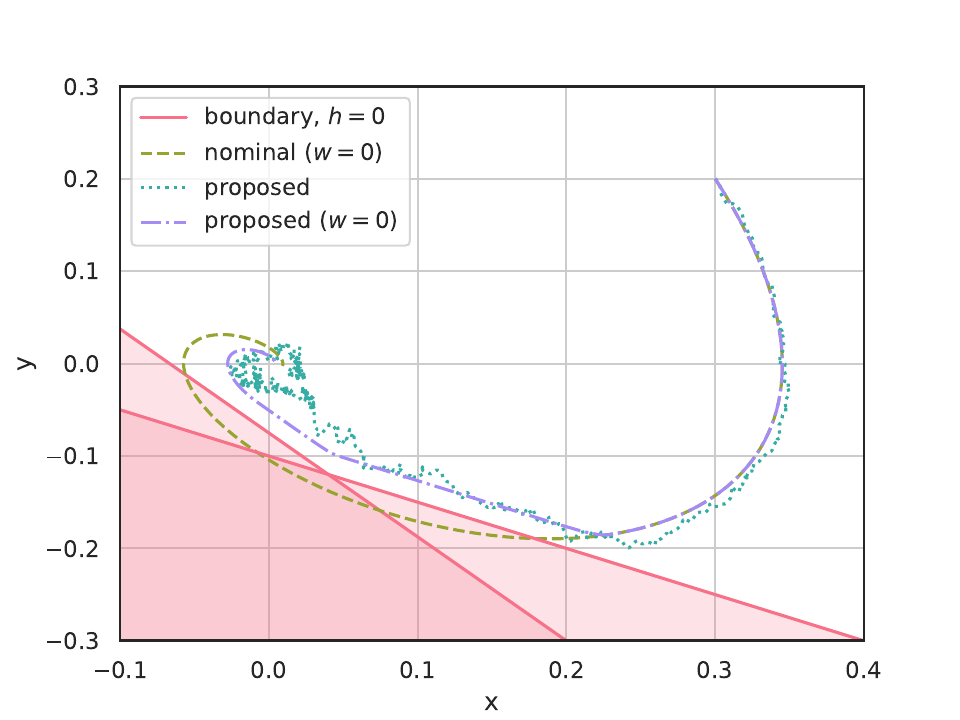}\vspace{-.1in}
\subcaption{Polytopic safe set} 
\label{fig:poly}
\end{minipage}
\begin{minipage}[b]{0.95\hsize}
 \centering
  \includegraphics[width=.98\linewidth, viewport =10 0 430 310, clip]{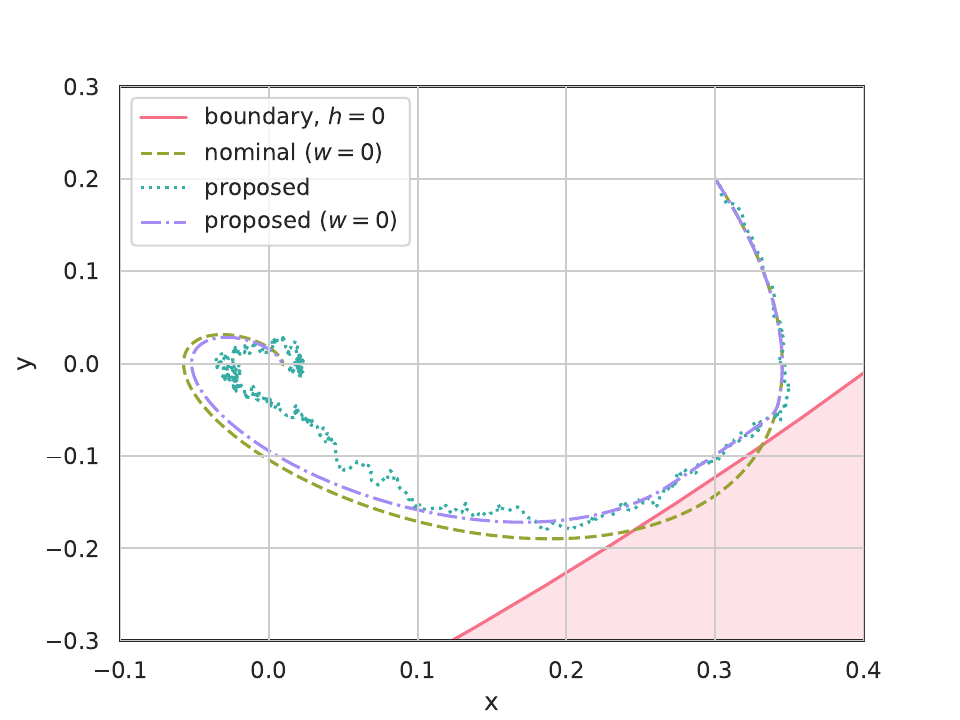}\vspace{-.1in}
\subcaption{Polytopic safe set} 
\label{fig:ell}
\end{minipage}
\caption{Phase portraits of performances of proposed controller. The area designated as unsafe is highlighted in red.}
\end{figure}

Next, the proposed controller is compared with the standard control-barrier function based controller for the polytopic safe set described above.
The standard controller basically sets $\sup_{\mathbb{P}\in \mathcal{P}}\mathbb{P}\text{-CVaR}_{\varepsilon}[Q^\top w ]=0$ in \eqref{eq:poly_c}, which is equivalent to using the expected value instead of the worst-case CVaR or disregarding the uncertainties in the controller design. The result is shown in Figure \ref{fig:polyb}. 
Since the trajectory of  ``Standard $w=0$'' case goes on the boundary, $h=0$,  the trajectory of ``Standard''  case enters the unsafe set quite often due to the disturbance. 
Thus, we see that considering the expected value for stochastic systems may not be suitable for safety-critical systems; the trajectory of the proposed approach may still enter the unsafe set, but its risk is quantified and limited in design.

\begin{figure}\centering
 \includegraphics[width=.98\linewidth, viewport =10 0 430 310, clip]{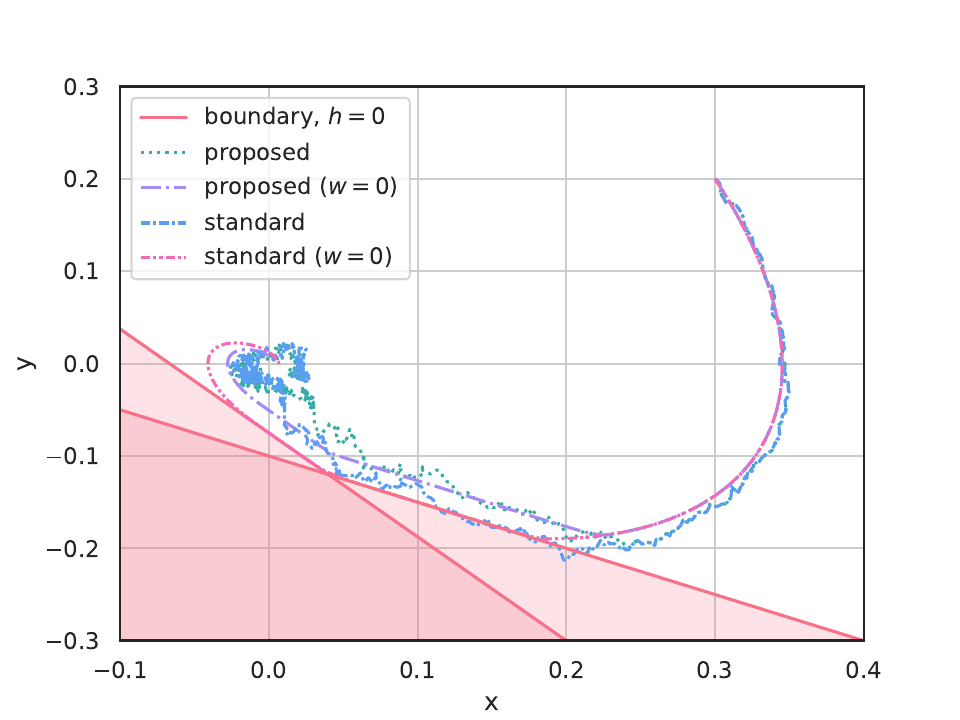}
\caption{Phase portraits of proposed controller and standard controller for polytopic safe set. The area designated as unsafe is highlighted in red.} 
\label{fig:polyb}\vspace{-.1in}
\end{figure}

\section{Conclusions}\label{sec:conc}
This paper presented a risk-aware control approach to enforce safety that integrates the worst-case CVaR with the control barrier function for discrete-time nonlinear systems with stochastic uncertainties.
The approach is based on some useful findings about the worst-case CVaR, and effectively integrates safe sets while considering the tail risk in the controller design. 
More specifically, we formulate specific computational problems to compute control inputs for three different safe sets: a quadratic program for half-space and polytopic sets and a semidefinite program for an ellipsoidal set.
Our validation with the inverted pendulum confirmed its effectiveness and demonstrated improved performance over existing methods. 
Future work should explore extensions of event- and self-triggered controllers to reduce resource usage as discussed in \cite{Kis23b}.

\addtolength{\textheight}{-7cm}   







\bibliographystyle{IEEEtran}
\bibliography{IEEEabrv,myref}

\end{document}